\documentclass[12pt,a4paper]{article}
\usepackage[left=2cm, right=2cm, top=2cm, bottom=2cm]{geometry}
\usepackage{graphicx}
\usepackage{mathtools}
\usepackage{amssymb}
\usepackage{amsthm}
\usepackage{thmtools}
\newtheorem{definition}{Definition}[section]
\newtheorem{proposition}{Proposition}[section]
\newtheorem{theorem}{Theorem}[section]
\newtheorem{example}{Example}[section]
\newtheorem{remark}{Remark}[section]
\newtheorem{lemma}{Lemma}[section]
\newtheorem{corollary}{Corollary}[section]
\usepackage{multirow}
\usepackage[english]{babel}
\begin{document}
	\begin{center}
		\textbf{\Large Some more constructions of $ n-$cycle permutation polynomials}
	\end{center}
	\begin{center}
		\textbf{Varsha Jarali$^1$, Prasanna Poojary$^2$ and Vadiraja Bhatta G. R.$^3$}\\[3pt]
		$^{1,3}$Department of Mathematics, Manipal Institute of Technology, Manipal Academy of Higher
		Education, Manipal 576104, India.\\[3pt]
		$^2$Department of Mathematics, Manipal Institute of Technology Bengaluru, Manipal Academy of
		Higher Education, Manipal 576104, India.\\[3pt]
		$^1$varshapjarali@gmail.com, $^2$poojary.prasanna@manipal.edu, $^3$vadiraja.bhatta@manipal.edu
	\end{center}
	\begin{abstract}
		$ n-$cycle permutation polynomials with small n have the advantage that their compositional inverses are efficient
in terms of implementation. These permutation polynomials have significant applications in cryptography and coding theory. In this article, we propose criteria for the construction of $ n-$cycle permutation using linearized polynomial $ L(x) $ for larger $ n $. Furthermore, we investigate and generalize certain novel forms of $ n-$cycle permutation polynomials. Finally, we demonstrate our approach by constructing explicit $ n-$cycle permutation of the form $ L(x)+\gamma h(Tr_{q^{m}/q}(x)) $, and $ G(x)+\gamma f(x) $ with a Boolean function $ f(x) $. The polynomial $ x^{d}+\gamma f(x) $ with $ f(x) $ being a Boolean function is shown to be quadruple and quintuple permutation polynomials. Moreover, linear binomial triple-cycle permutation polynomials are constructed.
	\end{abstract}
    \textbf{Mathematics Subject Classification (2020) 2010:}· 05A05 · 11T06.

\section{Introduction}
\label{intro}
Let $ F_{q} $ be the finite field with $ q $ elements, where $q$ is a power of a prime. A polynomial $ f(x)\in F_{q}[x] $ is called a permutation polynomial if the map $ f:c\rightarrow f(c) $ is a bijection on $ F_{q} $. The study of permutation polynomials has a long history, in the recent two decades, several types of research have been conducted over finite fields \cite{akbary2011constructing,li2018new,yuan2011permutation,yuan2014further,li2020some,tu2015two,zeng2015permutation,zieve2010classes,zieve2009some,marcos2011specific,vadiraja2017study,charpin2008class,jarali2022construction} and they have a wide range of applications in combinatorial designs \cite{ding2006family}, coding theory \cite{ding2014binary,ding2013cyclic,laigle2007permutation}, and cryptography \cite{rivest1978method,singh2009public,singh2020public,singh2011poly,khachatrian2017permutation}. We know that, permutation polynomials over a finite field $ F_{q} $ form a group under composition and reduction modulo $ x^{q}-x $ and that is isomorphic to the symmetric group $ S_{q} $ on $ q $ letters. Therefore, there is a unique compositional inverse  $ f^{-1} $ such that $ f\circ f^{-1}=f^{-1}\circ f= I $, where $ I $ is an identity map.  If there exists a positive integer $ n $ such that $ f^{(n)}=I $, then $ f $ is called $ n-$cycle permutation polynomial. Each permutation over the finite field must be an $ n-$cycle permutation polynomial for at least one positive integer $ n $. In this case, compositional inverse of $ f $ is $ f^{(n-1)}=f^{-1} $. For  $ n=2,3,4,... $ the $ n-$cycle permutation polynomials are named as involution, triple, quadruple, and quintuple respectively. 

In many cases, both the permutation polynomials and their compositional inverses are necessary. Permutation polynomials are often used in block ciphers to design substitution boxes (S-boxes), which form the confusion layer during the encryption process. The recent study \cite{hasan2021c} shows that involutions can be used to construct 4- uniform permutation polynomials that explicitly determine the $ c-$difference distribution table (DDT)  and boomerang connectivity table (BCT) entries, which play a major role in the design of S-boxes in block ciphers. Since involutions are self-inverses, they reduce the extra resources needed for decryption. Further, it is known that these are good candidates when it comes to cryptanalytical attacks as well. Several studies on involutions have been done; the explicit study of involutions was initiated by Charpin, Mesnager, and Sarkar \cite{charpin2016involutions,mesnager2016constructions,charpin2015involutions,charpin2016dickson}. They investigated the monomial involutions and their enumeration in addition to that they constructed linear involutions, linear binomials, and involutions added with a Boolean function over finite fields with even characteristics. Furthermore,  Zheng et al. \cite{zheng2019constructions} gave another criterion for characterization of involutions of the form $ x^{r}h(x^{s}) $ over $ F_{q} $, where $ s|(q-1) $. The AGW Criterion, introduced by Akbary et al. \cite{akbary2011constructing} in 2011, is a potential technique for building permutation polynomials across finite fields. Its importance stems from the fact that it is used to build numerous new classes of permutation polynomials in addition to providing a coherent explanation for numerous previous permutation polynomial constructions. Motivated by this, 
 Niu et al. \cite{niu2020new}  constructed some of the involutions over finite fields by proposing an involutory version of the AGW criterion. Liu et al. \cite{liu2019triple}  computed triple cycle permutation polynomials of monomials, Dickson polynomials, and linearized polynomials over a finite field $ F_{2^{m}} $. Recently,  Wu et al. \cite{wu2020characterizations} obtained characterizations of triple cycle permutation polynomials of the form $ x^{r}h(x^{s}) $. After that, Chen et al.  \cite{chen2021constructions} generalized previously known involutions and gave a unified criterion for the construction of  $ n-$cycle permutation of the form $ x^{r}h(x^{s}) $.  Very recently, Niu et al.  \cite{niu2023more} proposed some criteria for $ n-$cycle permutation polynomials of the form $ xh(\lambda(x)),h(\psi(x))\phi(x)+g(\psi(x))$, and $g(x^{q^{i}}-x+\delta)+bx $ with general $ n $. 

The study of  $ n-$cycle permutation polynomials is motivated by the criterion proposed for the construction of involutions \cite{charpin2016involutions,mesnager2016constructions,charpin2015involutions,charpin2016dickson} and $ n-$cycle permutation polynomials \cite{chen2021constructions,niu2023more}. The usage of the linear structure of a Boolean function for the construction of triple-cycle permutation polynomials \cite{liu2019triple} also inspired the study.  In this chapter, we propose some of the monomial $ n-$cycle permutation polynomials of the type $ x^{2^{2k}-2^{k}+1} $, $ x^{2^{k}+1} $ for Kasami and Gold parameters, respectively, and generalize the enumeration for $ n-$cycle monomials. Furthermore, we construct $ n-$cycle permutation polynomials using the linearized polynomial $L(x) $, and the polynoials $ L(x)+\theta h(Tr_{q^{m}/q}(x))$, and $ G(x)+\gamma f(x) $, where $ f(x) $ is a Boolean function over $ \mathbb{F}_{2^{m}} $ . Finally, additional $ n-$cycle permutation polynomials are derived for $ n<6 $. 

The rest of the paper is organized as follows. In Section 2, we introduce some of the basic definitions. In Section 3, we propose some of the monomial $ n-$cycle permutation polynomials of the form $ x^{2^{2k}-2^{k}+1} $, $ x^{2^{k}+1} $ for Kasami and Gold parameters, respectively, and generalize the enumeration for $ n-$cycle monomials. Furthermore, we construct $ n-$cycle permutation polynomials using the linearized polynomial $L(x) $, and the polynoials $ L(x)+\theta h(Tr_{q^{m}/q}(x))$, and $ G(x)+\gamma f(x) $, where $ f(x) $ is a Boolean function over $ \mathbb{F}_{2^{m}} $ . Finally, additional $ n-$cycle permutation polynomials are derived for $ n<6 $ in Section 4.

\section{Preliminary}
A polynomial $ L(x) $ with the coefficients in an extension field $F_{q^{m}}$ of the field $F_{q}$ represented as 
\begin{equation*}
	L(x)= \sum\limits_{i=0}^{m-1} \alpha_{i}x^{q^{i}},
\end{equation*}
and it is known as $ q- $ linearized polynomial over $F_{q^{m}}$.

Let $ m $ be a positive integer. We denote $ Tr_{q^{m}/q}(x) $ to be trace function from $ F_{q^{m}} $ to $ F_{q} $ such that 
\begin{equation}
	Tr_{q^{m}/q}(x)=x+x^{q}+x^{q^{2}}+\cdots+x^{q^{m-1}}.
\end{equation}

\begin{definition}\cite{niu2023more}\label{d1}
	If there exists a positive integer $ n $ such that $ f^{(n)}=I $, we say $ f $ is an $ n-$cycle permutation polynomial. 	
\end{definition}
One can also prove  $ f^{(n-1)}(x)=f^{-1}(x) $ for all $ x\in F_{q^{m}} $ to show $ f(x) $ is an $ n-$cycle permutation polynomial over $ F_{q^{m}} $.

\begin{definition}\label{d2}\cite{charpin2016involutions}
	Let $ m=rk $, $1\leq k\leq m$. Let $ f $ be a function from $ F_{2^{m}} $ to $ F_{2^{k}} $, $ \gamma\in F_{2^{m}}^{*} $ and $ b $ is fixed in $ F_{2^{k}} $. Then, $ \gamma $ is a $ b- $ linear structure of $ f $ if $ f(x)+f(x+u\gamma)=ub $ for all $ x\in F_{2^{m}} $ and for all $ u\in F_{2^{k}} $. In particular, when $ k=1 $, $\gamma$ is usually said to be a $ b- $ linear structure of the Boolean function $ f $, that is $ f(x)+f(x+\gamma)=b $ for all $ x\in F_{2^{m}} $. 
\end{definition}
\begin{lemma}\label{l2}\cite{charpin2008class}
	Let $ G(x) $ be a permutation polynomial of $ F_{2^{m}} $, $ f $ be a Boolean function on $ F_{2^{m}} $ and $ \gamma\in F_{2^{m}}^{*} $, then $ S(x)=G(x)+\gamma f(x) $ is a permutation polynomial of $ F_{2^{m}} $ if and only if $ \gamma $ is a 0- linear structure of $ f\circ G^{-1} $.
\end{lemma}
\begin{lemma}\label{l3}\cite{charpin2016involutions}
	Let $ G(x) $ be a permutation polynomial of $F_{2^{m}}$, $ f $ be a Boolean function on $ F_{2^{m}} $ and $ \gamma\in F_{2^{m}}^{*} $. Assume $ S(x)= G(x)+\gamma f(x)$ is a permutation polynomial, then $ S^{-1}=G^{-1}\circ F $, where $ F(x)=x+\gamma f(G^{-1}(x)) $.
\end{lemma}
\begin{theorem}\label{t1}\cite{wu2013linearized}
	Let $ L(x)=\sum\limits_{i=0}^{m-1}a_{i}x^{q^{i}} $ be a linearised permutation polynomial over $ F_{q^{m}} $ and $ \bar{a_{i}} $ denote the $(i,0)-$ th cofactor of $ D_{L} $ the determinant of $ D_{L} $ is, 
	\begin{equation}
		det(D_{L})=a_{0}\bar{a_{0}}+\sum\limits_{i=0}^{m-1}a^{2^{i}}_{m-i}\bar{a_{i}}.
	\end{equation}
	Then, the compositional inverse of $ L(x) $ is given by,
	\begin{equation}
		L^{-1}(x)=\frac{1}{det(D_{L})}\sum\limits_{i=0}^{m-1}\bar{a_{i}}x^{2^{i}}.
	\end{equation}
\end{theorem}

\section{$ n-$cycle permutation polynomials}\label{sec2}

In this section, we prove that $ L(x) $ is an $ n-$cycle permutation polynomial and we characterize some new $ n-$cycle permutation polynomials.The proof relies on showing the equality $ F^{(n-1)}(x)=F^{-1}(x) $. The derivation of $ L^{-1}(x)$ for $ L(x) $ is taken from Theorem (\ref{t1}).

\begin{proposition}\label{p11}
    The linearized polynomial $ L(x)=\sum\limits_{i=0}^{m-1}a_{i}x^{2^{i}} $ is a triple cycle permutation polynomial over $ F_{2^{m}} $ if and only if 
    \begin{equation}
		det(D_{L})\not=0 
  \end{equation}and
 \begin{equation}
       \bar{a_k}=det(D_L) \left(\sum\limits_{i=0}^{k}a_ia_{k-i}^{2^i}+\sum\limits_{i=k+1}^{m-1}a_ia_{m+1-i}^{2^i}\right)
    \end{equation}
    for each $ k\in \{0,1,\ldots,m-1\} $, where $ k-i $ is computed modulo $ m $.
\end{proposition}
\begin{proof}
    Consider linearized polynomial $ L(x)=\sum\limits_{i=0}^{m-1}a_{i}x^{2^{i}} $. Let us compose $L(x)$ two times,
    \begin{equation}
        L(L(x))=\sum\limits_{j=0}^{m-1}\sum\limits_{i=0}^{m-1}a_ia_j^{2^i}x^{2^{i+j}}=\sum\limits_{k=0}^{m-1}\sum\limits_{i=0}^{k}a_ia_{k-i}^{2^i}x^{2^k}+\sum\limits_{k=m}^{2m-2}\sum\limits_{i=k-m+1}^{m-1}a_ia_{k-i}^{2^i}x^{2^k}
    \end{equation}

    Here we use the property of the finite field $ F_{2^{m}} $, $x^{2^m}=x$ and further reduce the above equation \begin{equation}
        L(L(x))=\sum\limits_{j=0}^{m-1}\sum\limits_{i=0}^{m-1}a_ia_j^{2^i}x^{2^{i+j}}=\sum\limits_{k=0}^{m-1}\left(\sum\limits_{i=0}^{k}a_ia_{k-i}^{2^i}+\sum\limits_{i=k+1}^{m-1}a_ia_{m+1-i}^{2^i}\right)x^{2^k}
    \end{equation}
   
    Noting the fact that $ L(x) $ is a permutation polynomial if and only if the matrix of $ D_{L} $ is non-singular, $ L(x) $ is a triple cycle permutation polynomial if and only if $ det(D_{L})\not=0 $ and $ L^{2}(x)=L^{-1}(x) $. Further writing $L^{-1}(x) $ using Theorem (\ref{t1}) as,    
	\begin{equation}
		L^{-1}(x)=\frac{1}{det(D_{L})}\sum\limits_{i=0}^{m-1}\bar{a_{i}}x^{2^{i}}.
	\end{equation}
	Comparing the coefficients on both sides, we get
    \begin{equation}
       \bar{a_k}=det(D_L) \left(\sum\limits_{i=0}^{k}a_ia_{k-i}^{2^i}+\sum\limits_{i=k+1}^{m-1}a_ia_{m+1-i}^{2^i}\right)
    \end{equation}
    for each $ k\in \{0,1,\ldots,m-1\} $, where $ k-i $ is computed modulo $ m $.
\end{proof}

\begin{example}
    Consider the linearized polynomial $L(x)=a_0x+a_1x^2$ and $a_i=0$ for all $i=2,3,4,\ldots,m-1$, then 
    $$L(L(x))=a_0^2x+(a_0a_1+a_1a_0^2)x^2+a_1^3x^{2^2}.$$
    Hence, $L(x)=a_0x+a_1x^2$ is triple cycle permutation polynomial if and only if  \begin{equation}
		det(D_{L})\not=0 
  \end{equation}and
  \begin{equation}
       \bar{a_k}=det(D_L) \left(a_0^2+(a_0a_1+a_1a_0^2)+a_1^3\right)
    \end{equation}
    for each $ k\in \{0,1,\ldots,m-1\} $, where $ k-i $ is computed modulo $ m $.
\end{example}

\begin{theorem}\label{t2}
	The linearized polynomial $ L(x)=\sum\limits_{i=0}^{m-1}a_{i}x^{2^{i}} $ is an $ n-$cycle permutation polynomial over $ F_{2^{m}} $ if and only if 
	\begin{equation}
		det(D_{L})\not=0 
  \end{equation}and
 \begin{equation}
		\bar{a}_{i_{(2n-1)}}=det(D_{L})\sum\limits_{i_{(2n-1)}=0}^{m-1}c_{i_{(2n-1)}},
	\end{equation} 
     where,\begin{equation}
        c_{i_{(2n-1)}} = \begin{cases}
    \sum\limits_{i_{(2n-3)}=0}^{i_{(2n-1)}}c_{i_{(2n-3)}}a_{i_{(2n-1)}-i_{(2n-3)}}^{2^{i_{(2n-3)}}} & $\text{if $ 0\leq i_{(2n-3)}\leq i_{(2n-1)} $}$,\\
    \sum\limits_{i_{(2n-3)}=i_{(2n-1)}+1}^{m-1}c_{i_{(2n-3)}}a_{m+1-i_{(2n-3)}}^{2^{i_{(2n-3)}}} & $\text{if $i_{(2n-1)}+1\leq i_{(2n-3)}\leq m-1$}$.
		\end{cases}
    \end{equation}
  for each $ i_{(2n-1)}\in \{0,1,\ldots,m-1\} $, where $ i_{(2n-1)}-i_{(2n-3)} $ is computed modulo $ m $.
\end{theorem}
\begin{proof}
	 Consider linearized polynomial $ L(x)=\sum\limits_{i_1=0}^{m-1}a_{i_1}x^{2^{i_1}} $. Let us compose $L(x)$ two times,
    \begin{equation}
    \begin{split}
        L(L(x))&=\sum\limits_{i_2=0}^{m-1}\sum\limits_{i_1=0}^{m-1}a_{i_1}a_{i_2}^{2^{i_1}}x^{2^{i_1+i_2}}\\&=\sum\limits_{i_3=0}^{m-1}\sum\limits_{i_1=0}^{i_3}a_{i_1}a_{i_3-i_1}^{2^{i_1}}x^{2^{i_3}}+\sum\limits_{i_3=m}^{2m-2}\sum\limits_{i_1=i_3-m+1}^{m-1}a_{i_1}a_{i_3-i_1}^{2^{i_1}}x^{2^{i_3}}
    \end{split}
    \end{equation}

    Here we use the property of the finite field $ F_{2^{m}} $, $x^{2^m}=x$ and further reduce the above equation \begin{equation}\label{eq1}
        L(L(x))==\sum\limits_{i_3=0}^{m-1}\left(\sum\limits_{i_1=0}^{i_3}a_{i_1}a_{i_3-i_1}^{2^{i_1}}+\sum\limits_{i_1=i_3+1}^{m-1}a_{i_1}a_{m+1-i_1}^{2^{i_1}}\right)x^{2^{i_3}}
    \end{equation}
    
    Let us consider \begin{equation}
        c_{i_3} = \begin{cases}
    \sum\limits_{i_1=0}^{i_3}a_{i_1}a_{i_3-i_1}^{2^{i_1}} & $\text{if $ 0\leq i_1\leq i_3 $}$,\\
    \sum\limits_{i_1=i_3+1}^{m-1}a_{i_1}a_{m+1-i_1}^{2^{i_1}} & $\text{if $i_3+1\leq i_1\leq m-1$}$.
		\end{cases}
    \end{equation}
    for each $i_3\in \{0,1,\ldots,m-1\} $.

    Then above equation (\ref{eq1}) reduces to
    \begin{equation}
        L(L(x))=\sum\limits_{i_3=0}^{m-1}c_{i_3}x^{2^{i_3}}
    \end{equation}
    
	With one more composition, it follows that,
	\begin{equation}
		L(L(L(x)))=\sum\limits_{i_4=0}^{m-1}\sum\limits_{i_3=0}^{m-1}c_{i_3}(a_{i_4}x^{2^{i_4}})^{2^{i_3}}=\sum\limits_{i_4=0}^{m-1}\sum\limits_{i_3=0}^{m-1}c_{i_3}a_{i_4}^{2^{i_3}}x^{2^{i_4+i_3}}
	\end{equation}
    It reduces to
    \begin{equation}\label{eq2}
        \begin{split}
            L^3(x)=\sum\limits_{i_5=0}^{m-1}\left(\sum\limits_{i_3=0}^{i_5}c_{i_3}a_{i_5-i_3}^{2^{i_3}}+\sum\limits_{i_3=i_5+1}^{m-1}c_{i_3}a_{m+1-i_3}^{2^{i_3}}\right)x^{2^{i_5}}
        \end{split}
    \end{equation}
    
     Let us consider \begin{equation}
        c_{i_5} = \begin{cases}
    \sum\limits_{i_3=0}^{i_5}c_{i_3}a_{i_5-i_3}^{2^{i_3}} & $\text{if $ 0\leq i_3\leq i_5 $}$,\\
    \sum\limits_{i_3=i_5+1}^{m-1}c_{i_3}a_{m+1-i_3}^{2^{i_3}} & $\text{if $i_5+1\leq i_3\leq m-1$}$.
		\end{cases}
    \end{equation}
  for each $i_5=\in \{0,1,\ldots,m-1\} $.

    Then above equation (\ref{eq2}) reduces to
    \begin{equation}
        L^3(x)=\sum\limits_{i_5=0}^{m-1}c_{i_5}x^{2^{i_5}}
    \end{equation}
    
	In general for $(n-1)$ composition we get
	\begin{equation}
    L^{n-1}(x)=\sum\limits_{i_{(2n-1)}=0}^{m-1}c_{i_{(2n-1)}}x^{2^{i_{(2n-1)}}}
	\end{equation}
    where,\begin{equation}
        c_{i_{(2n-1)}} = \begin{cases}
    \sum\limits_{i_{(2n-3)}=0}^{i_{(2n-1)}}c_{i_{(2n-3)}}a_{i_{(2n-1)}-i_{(2n-3)}}^{2^{i_{(2n-3)}}} & $\text{if $ 0\leq i_{(2n-3)}\leq i_{(2n-1)} $}$,\\
    \sum\limits_{i_{(2n-3)}=i_{(2n-1)}+1}^{m-1}c_{i_{(2n-3)}}a_{m+1-i_{(2n-3)}}^{2^{i_{(2n-3)}}} & $\text{if $i_{(2n-1)}+1\leq i_{(2n-3)}\leq m-1$}$.
		\end{cases}
    \end{equation}
  for each $i_{(2n-1)}\in \{0,1,\ldots,m-1\} $.  
    
	By noting the fact that $ L(x) $ is a permutation polynomial if and only if matrix of $ D_{L} $ is non singular, $ L(x) $ is an $ n-$cycle permutation polynomial if and only if $ det(D_{L})\not=0 $ and $ L^{n-1}(x)=L^{-1}(x) $. Furthermore, using Theorem (\ref{t1}), $L^{-1}(x) $ can be expressed as:     
	\begin{equation}
		L^{-1}(x)=\frac{1}{det(D_{L})}\sum\limits_{k=0}^{m-1}\bar{a_{k}}x^{2^{k}}.
	\end{equation}
	By equating the coefficients on both sides, we obtain
	\begin{equation}
		\bar{a}_{i_{(2n-1)}}=det(D_{L})\sum\limits_{i_{(2n-1)}=0}^{m-1}c_{i_{(2n-1)}},
	\end{equation} 
	for each $ i_{(2n-1)}\in \{0,1,\ldots,m-1\} $, where $ i_{(2n-1)}-i_{(2n-3)} $ is computed modulo $ m $.
\end{proof}
\subsection{Monomials}
Inspired by \cite{niu2023more,cohn2012advanced} we consider the monomials in this section and discuss the results regarding the enumeration of $ n-$cycle permutation polynomials over $ F_{2^{m}} $.
\begin{lemma}\cite{niu2023more}\label{l1}
	Let $ f(x)=x^{d} $ be a monomial over $ F_{q^{m}} $. Then, $ f(x) $ is $ n-$cycle permutation polynomial over $ F_{q^{m}} $ if and only if $d^n \equiv 1 \pmod{q^m-1}$.
\end{lemma}
\begin{lemma}\cite{cohn2012advanced}\label{l4}
	Let $ m=\prod_{i=1}^{r}p_{i}^{\alpha_{i}} $ for $ r $ distinct prime factors $ p_{i} $'s	and $ r $ positive integers $ \alpha_{i} $'s. For each $ i $ with $1\leq i\leq r$, denote $ s_{i} $ the number of solutions to the congruence equation $ f(x)\equiv0\pmod {p_{i}^{\alpha_{i}}} $. Then, the congruence equation $ f(x)\equiv0\pmod m $ exactly has $ \prod_{i=1}^{r}s_{i} $ solutions.
\end{lemma}
We derive the following result using the same notations as in Lemma  (\ref{l4}).
\begin{proposition}
	The number of monomial $ n-$cycle permutation polynomials on $ F_{2^{m}} $ equals $ n^{t} $, where $ t $ is the number of all prime factors $ p_{i} $'s with $ n|(p_{i}-1) $ or $ \alpha_{i}\geq2  $ in the case $ p_{i}=n $.	
\end{proposition}
\begin{proof}
	We know that $ f(x)=x^{d} $ is $ n-$cycle permutation polynomial over $ F_{q^{m}} $ if and only if $ d^{n}\equiv1\pmod{ q^{m}-1} $. To determine the number of solutions to $ d^{n}\equiv1\pmod{ q^{m}-1} $, by Lemma (\ref{l4}) it is sufficient to find number of solutions to $ d^{n}\equiv1\pmod {p_{i}^{\alpha_{i}}} $ for each $ i $ with $1\leq i\leq r$. Let $ Z_{p_{i}^{\alpha_{i}}} $	be the residue class ring of integers modulo $p_{i}^{\alpha_{i}}$ and $ Z_{p_{i}^{\alpha_{i}}}^{*} $ be its multiplicative group consisting of all invertible elements in $ Z_{p_{i}^{\alpha_{i}}} $. Here all $ p_{i} $ is an odd as $ 2^{m}-1 $ is odd. As a result, there is a primitive root modulo $ p_{i}^{\alpha_{i}} $ and $ Z_{p_{i}^{\alpha_{i}}}^{*} $ is a cyclic group. Then each solution of  $ d^{n}\equiv1\pmod {p_{i}^{\alpha_{i}}} $ is exactly an element having order 1 or n in the group $ Z_{p_{i}^{\alpha_{i}}}^{*} $. Thus, the number of solutions of $ d^{n}\equiv1\pmod {p_{i}^{\alpha_{i}}} $  equals to $ n $ if $ n|\phi(p_{i}^{\alpha_{i}}) $ and 1 otherwise. 
\end{proof}
\begin{remark}
	If $ 2^{m}-1 $ is the prime number (Mersenne prime), the number of $ n-$cycle permutation polynomials having the form $ f(x)=x^{d} $ over $ F_{2^{m}} $ equals to 1 if $ n\not|(2^{m}-2) $ and $ n $ if $ n|(2^{m}-2) $.	
\end{remark}

Now we consider the monomials with varying  $ d $ over $ F_{2^{m}} $ and discuss the properties with respect to  $ n-$cycle permutation polynomials. The function $ f(x)=x^{2^{2k}-2^{k}+1} $ is a familiar one and is known as the Kasami power function.
\begin{proposition} For even $ m$,  $f(x)=x^{2^{2k}-2^{k}+1}  $ is an $ n-$cycle permutation polynomial over $ F_{2^{m}} $ if and only if $ 2^{m}-1 $ divides $ 2^{k}-1 $ i.e., $m$ devides $k$.
	
\end{proposition}
\begin{proof}
	Using the Lemma (\ref{l1}), we know that $ f(x)=x^{d} $ is an $ n-$cycle permutation polynomial over $ F_{2^{m}} $ if and only if $ d^{n}\equiv1\pmod{2^{m}-1} $. It is enough  to show that $ f(x)=x^{2^{2k}-2^{k}+1} $ is an $ n-$cycle permutation polynomial over $ F_{2^{m}} $ if and only if $ d=2^{2k}-2^{k}+1 $, whenever $ d^{n}\equiv1\pmod{2^{m}-1} $. This is possible if and only if $ 2^{m}-1 $ divides $ 2^{k}-1 $ i.e., $m$ devides $k$.
\end{proof}
In the following result, we consider the Gold power function $f(x)=x^{2^{k}+1}  $ where $\gcd(k,m)=1$.  The proof follows the same reasoning as in the earlier proof.

\begin{proposition}  $f(x)=x^{2^{k}+1}  $ is an $ n-$cycle permutation polynomial over $ F_{2^{m}} $ if and only if $m=1$.
	
\end{proposition}
\begin{proof}
    Using the Lemma (\ref{l1}), we know that $ f(x)=x^{d} $ is an $ n-$cycle permutation polynomial over $ F_{2^{m}} $ if and only if $ d^{n}\equiv1\pmod{2^{m}-1} $. Consider, $2^k+1\equiv 1 \pmod {2^m - 1} $ implies $2^k\equiv0 \pmod {2^m - 1}$. In other words, $2^k=t(2^m-1)$ for some positive integer $t$, which holds only when $m=1$. Hence, $f(x)=x^{2^{k}+1}  $ is an $ n-$cycle permutation polynomial over $ F_{2^{m}} $ if and only if $m=1$ for any positive integer $n$.
\end{proof}

Now, with the same technique used in \cite{niu2023more}, we prove $ F(x)=L(x)+\gamma h(Tr_{q^{m}/q}(x)) $ is an $ n-$cycle permutation polynomial.
\begin{corollary}\label{t3}
	Let linearized polynomial $ L(x)\in F_{q^{m}}[x] $ be an $ n-$cycle permutation polynomial over $ F_{q^{m}} $. Let $ h(x)\in F_{q}[x] $ and  $0\not=\gamma\in F_{q} $. The polynomial $ F(x)=L(x)+\gamma h(Tr_{q^{m}/q}(x)) $ satisfy condition $ Tr_{q^{m}/q}\circ F(x)=\bar{F}\circ Tr_{q^{m}/q}(x) $, with $ \bar{F}=L(x)+Tr_{q^{m}/q}(\gamma)h(x) $. Then, $ F(x) $ is a $ n -$ cycle permutation polynomial over $ F_{q^{m}} $ if and only if, 
	\begin{equation}\label{a1}
		\sum\limits_{i=0}^{m-1}L^{i}(h(\bar{F}^{(n-1-i)}(y)))=0 ,
	\end{equation}
	holds for any $ y \in Tr_{q^{m}/q}(F_{q^{m}}) $.
\end{corollary}
\begin{proof}
	For any $ x\in F_{q^{m}} $ and $ y \in Tr_{q^{m}/q}(F_{q^{m}}) $ with $ y=Tr_{q^{m}/q}(x) $ we have,
	
	\begin{equation}
		\begin{split}
			F^{2} & =L(F(x))+\gamma(h\circ Tr_{q^{m}/q}\circ F(x))\\
			&=L^{2}(x)+\gamma[L(h\circ Tr_{q^{m}/q}(x))+(h\circ\bar{F}\circ Tr_{q^{m}/q}(x))]\\
			&=\gamma \sum\limits_{i=0}^{1}L^{i}(h\circ \bar{F}^{(1-i)}\circ Tr_{q^{m}/q}(x))+L^{2}(x).
		\end{split}
	\end{equation}
	In general for $ n $ compositions we have,
	\begin{equation}\label{a2}
		\begin{split}
			F^{n}&=\gamma\sum\limits_{i=0}^{n-1}L^{i}(h\circ \bar{F}^{(n-1-i)}\circ Tr_{q^{m}/q}(x))+L^{n}(x)\\
			&=\gamma\sum\limits_{i=0}^{n-1}L^{i}(h( \bar{F}^{(n-1-i)}(y)))+x.
		\end{split}
	\end{equation} 
	Thus, if Equation (\ref{a1}) holds for any $ y\in Tr_{q^{m}/q}(F_{q^{m}}) $, then $ F $ is an $ n-$cycle permutation polynomial over $ F_{q^{m}} $. 
	
	Conversely, assuming that $ F(x) $ is an $ n-$cycle permutation polynomial over $ F_{q^{m}} $, for all $ y\in Tr_{q^{m}/q}(F_{q^{m}}) $ there exist $ x\in F_{q^{m}} $ such that $ Tr_{q^{m}/q}(x)=y $. Then, we  have $\gamma\sum\limits_{i=0}^{m-1}L^{i}(h( \bar{F}^{(n-1-i)}(y)))=0$, which implies $$ \sum\limits_{i=0}^{m-1}L^{i}(h(\bar{F}^{(n-1-i)}(y)))=0.$$
\end{proof}
\begin{proposition}\label{p1}
	Let   $ L_{1}(x)=\sum\limits_{i=0}^{m-1}a_{i}x^{q^{i}} $ and $ L_{2}(x)=\sum\limits_{j=0}^{m-1}b_{j}x^{q^{j}} \in F_{q^{m}}[x]$ be linearized polynomials and $ \gamma\in F_{q^{m}} $ with $ \gamma\not=0 $. If $Tr_{q^{m}/q}(L_{2}(x))=0 $, then $ F(x)=L_{1}(x)+\gamma L_{2}(Tr_{q^{m}/q}(x)) $ is an $ n-$cycle permutation polynomial over $ F_{q^{m}} $.  
\end{proposition}
\begin{proof}
	Comparing with the previous result we have  $ L_{1}(x)=L(x) $ and $ L_{2}(x)=h(x) $. Clearly, $ Tr_{q^{m}/q}\circ F(x)=\bar{F}\circ Tr_{q^{m}/q}(x) $ where $ \bar{F}=L_{1}(x)+Tr_{q^{m}/q}(\gamma L_{2}(x)) $. It is straightforward to verified that, if $ Tr_{q^{m}/q}(L_{2}(x))=0 $, then $ \bar{F}(x)=L_{1}(x) $. Similar verification implies that $ \bar{F}^{(n-1-k)}(x)=L_{1}^{(n-1-k)}(x) $.
	
	Thus, for any $ x\in F_{q^{m}} $, there exists $ y\in F_{q} $ such that $ Tr_{q^{m}/q}(x)=y $ and
	\begin{equation}
 \begin{split}
     \sum\limits_{k=0}^{m-1}L_{1}^{k}(L_{2}(L_{1}^{(n-1-k)}(y)))
		&=\sum\limits_{k=0}^{m-1}Tr_{q^{m}/q}(L_{2}(L_{1}^{(n-1)}(x)))\\
		&=\sum\limits_{k=0}^{m-1}L_{1}^{(n-1)}(Tr_{q^{m}/q}(L_{2}(x))\\
		&=0.
 \end{split}
	\end{equation}
	The above equality is true due to the fact that symbolic multiplication of linearized polynomials is commutative, and  if  $ Tr_{q^{m}/q}(L_{2}(x))=0 $. Thus, from Equation (\ref{a2}) $ F^{n}(x)=x $.
\end{proof}

In next theorem we prove $F(x)=G(x)+\gamma f(x)  $ is an $ n-$cycle permutation polynomial over $ F_{2^{m}} $, where $ f $ is a Boolean function over $ F_{2^{m}} $. Here we use the fact that,
any function $ f:F_{2^{m}}\rightarrow F_{2^{m}}  $ can be   represented uniquely as $ f(x)=\sum\limits_{i=0}^{2^{m}-1}a_{i}x^{i}$; $ a_{i},x\in F_{2^{m}} $ and $ f $ is a Boolean function if and only if $ f(x)^{2}=f(x) ,\forall x\in F_{2^{m}}$. The linear property of a Boolean function is also made use of in the result. 
\begin{theorem}\label{t4}
	Let $ G(x) $ be an $ n-$cycle permutation polynomial over $ F_{2^{m}} $ and $ f $ be a Boolean function with $ f=f\circ G $, $\gamma\in F_{q^{m}}^{*}$. Then, $ F(x)=G(x)+\gamma f(x) $ is an $ n-$cycle permutation polynomial over $ F_{2^{m}} $ if and only if the following two conditions are satisfied:
	\begin{enumerate}
		\item $\gamma$ is a $ 0-$linear structure of $ f $;
		\item $ G^{(n-1)}(x+\gamma)=G(G(\cdots(G(G(x)+\gamma )+\gamma )\cdots)+\gamma )+\gamma  $ for each \\$x\in S_{1}=\{x\in F_{2^{m}}|f(x)=1\}$.
	\end{enumerate}
\end{theorem}
\begin{proof}
	If $ f\circ G=f $ and  $ \gamma $ is a $ 0-$linear structure of a Boolean function  $ f $, then $ f(x)+f(x+\gamma)=0 $. Being a Boolean function  $ f(x) $ is linear over $F_{2^{m}}$. Using the fact that,
	any function $ f:F_{2^{m}}\rightarrow F_{2^{m}}  $ can be   represented uniquely as $ f(x)=\sum\limits_{i=0}^{2^{m}-1}a_{i}x^{i} ; a_{i},x\in F_{2^{m}} $, and $ f $ is a Boolean if and only if $ f(x)^{2}=f(x) ,\forall~ x\in F_{2^{m}}$. Hence,
	\begin{equation}
		F\circ F(x)= G[G(x)+\gamma f(x)]+\gamma f[G(x)+\gamma f(x)],
	\end{equation}
    It can be verified that
	\begin{equation}\label{eq11}
		F\circ F(x)=\begin{cases}
			G(G(x)) & \text{if } x\in F_{2^{m}}\backslash S_{1},\\
			G(G(x)+\gamma)+\gamma  f(G(x)+\gamma) & \text{if } x\in S_{1}.	
		\end{cases}
	\end{equation}
    $\gamma$ is $0-$linear structure of $f$ implies $f(G(x)+\gamma)+f(G(x))=0$ for $x\in S_1$. Hence equation (\ref{eq11}) reduces to,
    \begin{equation}\label{eq12}
		F\circ F(x)=\begin{cases}
			G(G(x)) & \text{if } x\in F_{2^{m}}\backslash S_{1},\\
			G(G(x)+\gamma)+\gamma   & \text{if } x\in S_{1}.	
		\end{cases}
	\end{equation}
	Similarly,
	\begin{equation}
		F\circ F\circ F(x)=G(G(G(x)+\gamma f(x))+\gamma f(x))+\gamma f(x).
	\end{equation}
	\begin{equation}
		F\circ F\circ F(x)=\begin{cases}
			G(G(G(x))) & \text{if } x\in F_{2^{m}}\backslash S_{1},\\
			G(G(G(x)+\gamma)+\gamma)+\gamma & \text{if } x\in S_{1}.
		\end{cases}
	\end{equation}
	Proceeding further similarly,
	\begin{equation}
		F^{(n-1)}(x)=G\{G(\cdots(G(G(x)+\gamma f(x) )+\gamma f(x) )\cdots)+\gamma f(x)\}+\gamma f(x). 
	\end{equation}
	\begin{equation}
		F^{(n-1)}(x)=\begin{cases}
			G^{(n-1)}(x) & \text{if } x\in F_{2^{m}}\backslash S_{1},\\
			G(G(G\cdots(G(G(x)+\gamma )+\gamma )\cdots+\gamma)+\gamma )+\gamma & \text{if } x\in S_{1}.
		\end{cases}
	\end{equation}
	By the Lemma (\ref{l2}), if  $ \gamma  $ is a $0-$linear structure of $ f $   implies that $ F(x) $ is a permutation polynomial. And by the Lemma (\ref{l3}),
    
    $ F^{-1}(x) $ does exist and is represented by  $ F^{-1}=G^{-1}\circ S $, where $ S=x+\gamma f(G^{-1}(x)) $. Thus,
	\begin{equation}
		F^{-1}=G^{-1}(x+\gamma f(G^{-1}(x)))\\
		=G^{-1}(x+\gamma f).
	\end{equation}
	That can be further simplified as,
	\begin{equation}
		F^{-1}=\begin{cases}
			G^{-1}(x) & \text{if } x\in F_{2^{m}}\backslash S_{1},\\
			G^{-1}(x+\gamma) & \text{if } x\in S_{1}.
		\end{cases}
	\end{equation}
	Since $ G(x) $ is an $ n-$cycle permutation polynomial, it follows that $ G^{(n-1)}(x)=G^{-1}(x) $. After substituting this, the above form of $F^{-1}$ simplifies to:
	\begin{equation}
		F^{-1}=\begin{cases}
			G^{(n-1)}(x) & \text{if } x\in F_{2^{m}}\backslash S_{1},\\
			G^{(n-1)}(x+\gamma) & \text{if } x\in S_{1}.
		\end{cases}
	\end{equation}
	If $G^{(n-1)}(x+\gamma)=G(G(\cdots(G(G(x)+\gamma )+\gamma )\cdots)+\gamma )+\gamma  $ for each $x\in S_{1}$, then $ F^{(n-1)}(x)=F^{-1}(x) $.
	
	Conversely, if $ F(x) $ is an $ n-$cycle permutation polynomial over $ F_{2^{m}} $ by the Lemma (\ref{l2}), $ \gamma $ is a $0-$linear structure of $ f\circ G^{-1} $. Since $ G $ is an $ n-$cycle permutation polynomial over $ F_{2^{m}} $, we have $ f\circ G=f $. Then, $ f\circ G^{-1}=f\circ G ^{(n-1)}=f\circ G \circ G^{(n-2)}=f\circ G^{(n-2)}=\ldots=f $. Thus, if $ \gamma $ is a $0-$linear structure of $ f(x) $,  then $ f(G(x))+f(G(x)+\gamma)=0 $ for any $ x\in S_{1} $. As a result, we obtain
	\begin{equation}
		\begin{split}
			&G\{G(\cdots(G(G(x)+\gamma f(x) )+\gamma f(x) )\cdots)+\gamma f(x)\}+\gamma f(x)=\\ &G(G(\cdots(G(G(x)+\gamma )+\gamma )\cdots)+\gamma )+\gamma.
		\end{split}
	\end{equation}
	Since $ F(x) $ is an $ n-$cycle permutation polynomial, $ F^{(n-1)}(x)=F^{-1}(x) $. This further implies that $G^{(n-1)}(x+\gamma)=G(G(\cdots(G(G(x)+\gamma )+\gamma )\cdots)+\gamma )+\gamma  $ for each $x\in S_{1}$.
\end{proof}
\begin{remark}
	If the second condition in the above theorem is replaced by $ G(x+\gamma)=G(x)+\gamma $, then  $ F(x)=G(x)+\gamma f(x) $ can not be an $ n-$cycle permutation polynomial over $ F_{2^{m}} $.
\end{remark}
\section{Some more $ n-$cycle permutation polynomials}
In this section, we consider $ n-$cycle permutation polynomials for $ n<6 $. 
\begin{proposition}\label{c1}
	Let linearised polynomial $ L(x)\in F_{q^{m}}[x] $ be an involution over $ F_{q^{m}} $, and $ h(x)\in F_{q}[x] $. Let $ 0\not=\gamma\in F_{q} $ and $ F(x)=L(x)+\gamma h(Tr_{q^{m}/q}(x)) $ satisfying $ Tr_{q^{m}/q}\circ F(x)=\bar{F}\circ Tr_{q^{m}/q}(x) $, with $ \bar{F}=L(x)+Tr_{q^{m}/q}(\gamma)h(x) $, then $ F(x) $ is an involution over $ F_{q^{m}} $ if $ Tr_{q^{m}/q}(F_{q^{m}})\subseteq Ker(h) $.
\end{proposition}

\begin{proof}
	\begin{equation}
		\begin{split}
			F\circ F &=F[L(x)+\gamma (h\circ Tr_{q^{m}/q}(x))]\\
			&=L^{2}(x)+\gamma L(h(Tr_{q^{m}/q}(x)))+\gamma h(Tr_{q^{m}/q}(F(x)))\\
			&=x+\gamma[L(h(Tr_{q^{m}/q}(x)))+h(Tr_{q^{m}/q}(F(x)))].
		\end{split}
	\end{equation}
	If $ Tr_{q^{m}/q}(F_{q^{m}})\subseteq Ker(h) $, then $ F\circ F(x)=x $.
\end{proof}
In the next results, we prove that $F(x)=x^{d}+\gamma f(x)$ is a quadruple and quintuple over $ F_{2^{m}} $ by considering $ G(x)=x^{d} $ and assuming that it is $ 4^{th} $ and $ 5^{th} $ cycle permutation polynomial respectively.
\begin{proposition}\label{c2}
	Let $ \gamma\in F_{2^{m}}^{*}$, $ d^{4}\equiv1(\mod2^{m}-1) $ and $ f $ be the Boolean function over $ F_{2^{m}} $ such that $ f(x)=f(x^{d}) $ for all $ x\in F_{2^{m}} $. Then, $F(x)=x^{d}+\gamma f(x)$ is a quadruple permutation polynomial over $ F_{2^{m}} $ if and only if
	\begin{enumerate}
		\item $\gamma$ is a $0-$linear structure of $ f(x) $;
		\item $ \gamma^{d^{3}}+\gamma^{d^{2}}+\gamma^{d}+\gamma=0 $ and\\
		$ \sum\limits_{0<j<d^{2}}\gamma^{d^{2}-j}x^{dj}+\sum\limits_{0<j<d}\gamma^{dj}x^{d^{2}j} +\sum\limits_{0<j<d}\gamma^{dj+d-j}+\sum\limits_{0<j<d}\sum\limits_{0<k<d^{2}}\gamma^{d-j+dj-k}x^{dk}\\=\sum\limits_{0<j<d^{3}}\gamma^{d^{3}-j}x^{j}$.
	\end{enumerate} 
\end{proposition}
\begin{proof}
	The congruence	$ d^{4}\equiv1(\mod2^{m}-1) $ implies that $ \gcd(2^{m}-1,d)=1 $. Hence $ d^{-1}(\mod  2^{m}-1 ) $ exists  and $ G(x) $ is a permutation polynomial over $ F_{2^{m}} $, which means that $ d^{3}=d^{-1}(\mod2^{m}-1) $. Then we have,
	
	\begin{equation}
		f(x^{d^{-1}})=f(x^{d^{(n-1)}})=f(x^{d})=f(x).
	\end{equation}
	From Theorem (\ref{t1})
	\begin{equation}
		F^{3}(x)=G(G(G(x)+\gamma f(x))+\gamma f(x))+\gamma f(x).
	\end{equation}
	By substituting $ G(x)=x^{d} $ we get,
	\begin{equation}
		F^{3}(x)=(((x)^{d}+\gamma f(x))^{d}+\gamma f(x))^{d}+\gamma f(x).
	\end{equation}
	Using the expression, $ F^{2}(x)=(x^{d}+\gamma f(x))^{d}+\gamma f(x^{d}+\gamma f(x)) $ we get,
	\begin{equation}
		F^{3}(x)=[(x^{d}+\gamma f(x))^{d}+\gamma f(x^{d}+\gamma f(x))]^{d}+\gamma f[(x^{d}+\gamma f(x))^{d}+\gamma f(x^{d}+\gamma f(x))].\\
	\end{equation}
	If $ x\in F_{2^{m}}\backslash S_{1} $, we have $ f(x^{d}+\gamma f(x))=0 $ which gives $ F^{3}(x)=x^{d^{3}} $.
	
	If $ x\in S_{1} $, then $ f(x+\gamma)=1 $. Moreover, since $ G(x)=x^{d} $ is a quadruple, we have $ f=f\circ G $ and $ f(x)=f(x^{d}) $ which imply $ \gamma $ is a $0-$linear structure of $ f $. Then, $ f(x^{d}+\gamma)=1 $.
	
	Now,
	\begin{equation}
		\begin{split}
			F^{3}(x)&=(x^{d}+\gamma)^{d^{2}}+\sum\limits_{0<j<d}(\gamma^{d-j}(x^{d}+\gamma)^{dj})+\gamma^{d}+\gamma\\
			&=x^{d^{3}}+\gamma^{d^{2}}+\gamma^{d}+\gamma+\sum\limits_{0<j<d^{2}}\gamma^{d^{2}-j}x^{dj}+\sum\limits_{0<j<d}\gamma^{dj}x^{d^{2}j}\\ &+\sum\limits_{0<j<d}\gamma^{dj+d-j}+\sum\limits_{0<j<d}\sum\limits_{0<k<d^{2}}\gamma^{d-j+dj-k}x^{dk}.
		\end{split}
	\end{equation}
	By Lemma (\ref{l2}), $ F(x) $ is a permutation polynomial over $ F_{2^{m}} $ if and only if $ \gamma $ is a $0-$linear structure of $ f $. Then, by Lemma (\ref{l3}) $ F^{-1}(x) $ can be expressed as, 
	\begin{equation}
		\begin{split}
			F^{-1}(x)&=(x+\gamma f(x^{d^{-1}}))^{d^{-1}}\\
			&=(x+\gamma f(x^{d^{3}}))^{3}\\
			&=x^{d^{3}}+\sum\limits_{0<j<d^{3}}(\gamma^{d^{3}-j}x^{j})f(x)+\gamma^{d^{3}}f(x).
		\end{split}
	\end{equation}
	On the other hand,
	\begin{equation}
		F^{-1}(x)=\begin{cases}
			x^{d^{3}} & \text{if } x\in F_{2^{m}}\backslash S_{1},\\
			x^{d^{3}}+\sum\limits_{0<j<d^{3}}(\gamma^{d^{3}-j}x^{j})+\gamma^{d^{3}} & \text{if } x\in S_{1}.
		\end{cases}
	\end{equation}
	Thus, $ F^{3}(x)=F^{-1}(x) $ if and only if $ \gamma^{d^{3}}+\gamma^{d^{2}}+\gamma^{d}+\gamma=0 $ and $ \sum\limits_{0<j<d^{2}}\gamma^{d^{2}-j}x^{dj}+\sum\limits_{0<j<d}\gamma^{dj}x^{d^{2}j} +\sum\limits_{0<j<d}\gamma^{dj+d-j}+\sum\limits_{0<j<d}\sum\limits_{0<k<d^{2}}\gamma^{d-j+dj-k}x^{dk}=
 \sum\limits_{0<j<d^{3}}\gamma^{d^{3}-j}x^{j}$.
\end{proof}
In a similar manner, the following result can be proved.
\begin{proposition}\label{c3}
	Let $ \gamma\in F_{2^{m}}^{*}$, $ d^{5}\equiv1(\mod2^{m}-1) $ and $ f $ be the Boolean function over $ F_{2^{m}} $ such that $ f(x)=f(x^{d}) $ for all $ x\in F_{2^{m}} $. Then, $F(x)=x^{d}+\gamma f(x)$ is a quintuple permutation polynomial over $ F_{2^{m}} $ if and only if
	\begin{enumerate}
		\item $\gamma$ is a $0-$linear structure of $ f(x) $;
		\item $\gamma^{d^{4}}+ \gamma^{d^{3}}+\gamma^{d^{2}}+\gamma^{d}+\gamma=0 $ and\\ $ \sum\limits_{0<l<d^{3}}\gamma^{d^{3}-l}x^{dl}+\sum\limits_{0<k<d^{2}}x^{d^{2}k}\gamma^{d^{2}-k}+\sum\limits_{0<k<d^{2}}\gamma^{dk+d^{2}-k}+
  \sum\limits_{0<j<d}\gamma^{d-j}x^{d^{3}j}
		+\\\sum\limits_{0<j<d}\gamma^{d^{2}j+d-j}+
  \sum\limits_{0<j<d}\gamma^{dj+d-j}+
  \sum\limits_{{0<j<d},{0<l<d^{3}}}x^{dl}\gamma^{d^{3}+d-l-j}
		+\\\sum\limits_{{0<j<d},{0<k<d^{2}}}x^{d^{2}k}\gamma^{d^{2}+d-j-k}+
		\sum\limits_{{0<j<d},{0<k<d^{2}}}\gamma^{dk+d^{2}-k+d-j}+\\
		\sum\limits_{{0<k<d^{2}},{0<l<d^{3}}}x^{dl}\gamma^{d^{3}+d^{2}-l-k}+
		\sum\limits_{{0<j<d},{0<k<d^{2}},{0<l<d^{3}}}x^{dl}\gamma^{d^{3}-l+d^{2}-k+d-j}=\\\sum\limits_{0<j<d^{4}}\gamma^{d^{4}-j}x^{j}$.
	\end{enumerate} 
\end{proposition}
We consider linear binomial and prove that it is a triple cycle permutation polynomial over $ F_{2^{m}} $.
\begin{theorem}\label{t5}
	Let $ F(x)=ax^{2^{i}}+bx^{2^{j}} $, $ a\in F_{2^{m}}^{*} $, and $ b\in F_{2^{m}}^{*} $, where $ i<j<m $. Then either of the following will hold;
	\begin{enumerate}
		\item For $ m $ relatively prime with both 2 and 3, $ F(x) $ can never be a triple cycle permutation polynomial.
		\item For $ m=3k $, or $ 2m=3k $, $ F(x) $ is a triple cycle if and only if $ j=i+k $ and
		
		either,
		\begin{equation}
			i=0\hspace{.2cm} and \hspace{.2cm}  a^{2}+b^{2}=1 \hspace{.2cm} and \hspace{.2cm} a^{2}b=0,  
		\end{equation}
		or, 
		\begin{equation}
			3i=2k,\hspace{.2cm}  a^{2^{2k}+1}=b^{2^{2k}+1},  ~a=a^{2^{i}}b^{2^{i}}(a^{2^{i}}+b^{2^{i}}) \hspace{.2cm} and\hspace{.2cm}  a^{2}+ab=1,
		\end{equation}
		
		or, 
		\begin{equation}
			3i=k,\hspace{.2cm}  a^{2^{k}+1}=b^{2^{k}+1}, ~ b=a^{2^{i}}b^{2^{i}}(a^{2^{i}}+b^{2^{i}}) \hspace{.2cm} and\hspace{.2cm}  b^{2}+ab=1. 
		\end{equation}
		\item For $ m=2k $, $ F(x) $ is a triple cycle if and only if $ i=j+k $ and 
		
		either,
		\begin{equation}
			i=0,\hspace{0.1cm}  a\in F_{2^{k}}, ~ b\in F_{2^{k}}, ~a^{2}+ab=1,\hspace{0.1cm} and\hspace{0.1cm}  ab+b^{3}=0,
		\end{equation}
		or,
		\begin{equation}
			j=0,\hspace{.2cm}  a\in F_{2^{k}} , b\in F_{2^{k}} \hspace{.2cm} and\hspace{.2cm}  ab+b^{2}=1.
		\end{equation}
	\end{enumerate}
\end{theorem}
\begin{proof}
	\begin{equation}
		\begin{split}
			F^{2}(x)&=a(ax^{2^{i}}+bx^{2^{j}})^{2^{i}}+b(ax^{2^{i}}+bx^{2^{j}})^{2^{j}}\\
			&=a^{2^{i}+1}x^{2^{2i}}+b^{2^{j}+1}x^{2^{2j}}+x^{2^{(i+j)}}(ab^{2^{i}}+a^{2^{j}}b)
		\end{split}
	\end{equation}
	\begin{equation}
		\begin{split}
			F^{3}(x)&=a^{2^{i}+1}(ax^{2^{i}}+bx^{2^{j}})^{2^{2i}}+b^{2^{j}+1}(ax^{2^{i}}+bx^{2^{j}})^{2^{2j}}+\\
   &(ax^{2^{i}}+bx^{2^{j}})^{2^{(i+j)}}(ab^{2^{i}}+a^{2^{j}}b)\\
			&=a^{2^{3i}+1}x^{2^{3i}}+b^{2^{3j}+1}x^{2^{3j}}+x^{2^{(2i+j)}}(a^{2^{i}+1}b^{2^{2i}}+a^{2^{(i+j)}+1}b^{2^{i}}+a^{2^{(i+2j)}}b)\\
			&+x^{2^{(i+2j)}}(a^{2^{2j}}b^{2^{j}+1}+ab^{2^{(2i+j)}}+a^{2^{j}}b^{2^{(i+j)}+1})	
		\end{split}	
	\end{equation}
	The exponents $ e $ of $ x $  belong to the set $\{3i,3j,2i+j,i+2j\}$. To get $ F^{3}(x)=x $ (noting that $ i\not=j $) it is required to remove three out of the four values from the set, which is explained in the following 6 cases.
	\begin{enumerate}
		\item  $ 3i\equiv3j(\mod m) $ implies $ 3i=3j+m $, which is impossible if $ 3\not|m $. On the other hand if $  m=3k $ then $ i=j+k $.
		\item $ 3i\equiv2i+j(\mod m) $ implies $ 3i=2i+j+m $ and hence $ i=j+m $ which is impossible. 
		\item $ 3i\equiv2j+i(\mod m) $ implies $ 3i=2j+i+m $ and hence $ 2i=2j+m $ which is impossible for $2\not|m$. On the other hand if $ m=2k $ then $ 2i=2j+2k $ implying  $ i=j+k $.
		\item $ 3j\equiv2i+j(\mod m) $ implies $ 3j=2i+j+m $ which is impossible if $2\not|m$. On the other hand if $ m=2k $ then $ 2j=2i+2k $ implying $ j=i+k $.
		\item $ 3j\equiv i+2j(\mod m) $ implies  $ 3j=2j+i+m $ implying $ j=i+m $ which is impossible.
		\item $ 2i+j\equiv i+2j(\mod m) $ implies $ 2i+j=i+2j+m $ and hence $ i=j+m $ which is impossible.
	\end{enumerate}
	When $ i=j+k $ and $ m=3k $ the form of  $ F^{3}(x) $ will be
	\begin{equation}
		\begin{split}
			F^{3}(x)&=(a^{2^{3i}+1}+b^{2^{3j}+1})x^{2^{3i}}+(a^{2^{i}+1}b^{2^{2i}}+a^{2^{i+j}+1}b^{2^{i}}+a^{2^{i+2j}}b)x^{2^{2i+j}}+\\
			&x^{2^{i+2j}}(a^{2^{2j}}b^{2^{j}+1}+ab^{2^{2i+j}}+a^{2^{j}}b^{2^{i+j}+1}).
		\end{split}
	\end{equation}
	The case  $ 3i=m $ possible only when $ i=0 $ and in this case, $ F^{3}(x)=x $ for all $ x $ if and only if $ a^{2}+b^{2^{3k}+1}=a^{2}+b^{2}=1$, $ a^{2}b+a^{2^{k}+1}b+a^{2^{2k}}b=0 $ and $ b^{2^{k}+1}a^{2^{2k}}+a^{2^{k}}b^{2^{k}+1}+ab^{2^{k}}=0 $. This is possible if $ a^{2}+b^{2}=1  $ and $ a^{2}b=0   $. Otherwise, we must have either $ 2i+j=m $ or $ i+2j=m $. If $ 2i+j=m $, we have  $ 3i=2k $ as $ j=k+i $ and in this case, $ F^{3}(x)=x $ for all $ x $ if and only if 
	\begin{equation}
		a^{2^{2k}+1}+b^{2^{2k}+1}=0,\hspace{.1cm} a^{2^{i}+1}b^{2^{2i}}+a^{2^{2i}+1}b^{2^{i}}+ab=1\hspace{.2cm} and 
	\end{equation}
	\begin{equation}
		b^{2^{i}+1}a^{2^{2i}}+a^{2^{i}}b^{2^{2i}+1}+ab=0.
	\end{equation}
	This is possible if $ a^{2^{2m}+1}=b^{2^{2m}+1} $, $ a=a^{2^{i}}b^{2^{i}}(a^{2^{i}}+b^{2^{i}}) $ and $ a^{2}+ab=1 $.

	Similarly, if $ i+2j=m $, we have  $ 3i=k $ as $ j=i+k $ and in this case, $ F^{3}(x)=x $ for all $ x $ if and only if
	\begin{equation}
		a^{2^{k}+1}+b^{2^{k}+1}=0,\hspace{.1cm} a^{2^{i}+1}b^{2^{2i}}+a^{2^{2i}+1}b^{2^{i}}+ab=0\hspace{.2cm} and 
	\end{equation}
	\begin{equation}
		b^{2^{i}+1}a^{2^{2i}}+a^{2^{i}}b^{2^{2i}+1}+ab^{2^{2m}}=1.
	\end{equation}
	This is possible if $ a^{2^{k}+1}=b^{2^{k}+1} $, $ b=a^{2^{i}}b^{2^{i}}(a^{2^{i}}+b^{2^{i}}) $ and $ b^{2}+ab=1 $.

	Now it remains to look at the $3^{rd}$ and $4^{th}$ cases. In both these cases  $ F(x) $ is a triple cycle permutation polynomial and when $ i=j+k $ with $ m=2k $, we have 
	\begin{equation}
		\begin{split}
			F^{3}(x)&=(a^{2^{3i}+1}+a^{2^{2j}}b^{2^{j}+1}+ab^{2^{2i+j}}+a^{2^{j}}b^{2^{i+j}+1})x^{2^{3i}}+\\
			&(b^{2^{3i}+1}+a^{2^{i}+1}b^{2^{2i}}+a^{2^{i+j}+1}b^{2^{i}}+a^{2^{i+2j}}b)x^{2^{3j}}.
		\end{split}
	\end{equation}
	The case $ 3i=m $ is possible only when $ i=0 $. 
	
	Consider two terms
	\begin{equation}\label{e1}
		a^{2}+b^{2^{k}+1}a+ab^{2^{k}}+a^{2^{k}}b^{2^{k}+1}\hspace{.2cm}and
	\end{equation}
	\begin{equation}\label{e2}
		b^{2^{k}+2}+a^{2}b+a^{2^{k}+1}b+ab,
	\end{equation}
	and if $ a\in F_{2^{k}} $, and $ b\in F_{2^{k}} $ Equations (\ref{e1}) and (\ref{e2}) reduce to, 
	\begin{equation}
		a(a+b^{2^{k}})+b(ab^{2^{k}}+a^{2^{k}}b^{2^{k}})=a^{2}+ab,
	\end{equation}
	\begin{equation}
		b(b^{2^{k}+1}+a^{2}+a^{2^{k}+1}+a)=ab+b^{3}.
	\end{equation}
	
	In this case, $ F^{3}(x)=x $ for all $ x $ if and only if $ a\in F_{2^{k}} $, $ b\in F_{2^{k}} $, $ a^{2}+ab=1 $, and $ ab+b^{3}=0 $.
	Otherwise, we must have $ 3j=m $ and is only when $ j=0 $.
	
	Consider two terms 
	\begin{equation}\label{e3}
		a^{2^{k}+2}+b^{2}a+ab+ab^{2^{k}+1}\hspace{.2cm} and
	\end{equation}
	\begin{equation}\label{e4}
		b^{2}+a^{2^{k}+1}b+a^{2^{k}+1}b^{2^{k}}+a^{2^{k}}b.
	\end{equation}
	and if $ a\in F_{2^{k}} $, and $ b\in F_{2^{k}} $ Equations (\ref{e3}) and (\ref{e4}) reduce to
	\begin{equation}
		a(a^{2^{k}+1}+b^{2}+b+b^{2^{k}+1})=0,
	\end{equation}
	\begin{equation}
		b^{2}+a^{2}b+a^{2}b+ab=ab+b^{2}.
	\end{equation}
	In this case, $ F^{3}(x)=x $ for all $ x $ if and only if $ a\in F_{2^{k}} $, $ b\in F_{2^{k}} $ and $ ab+b^{2}=1 $.
	Similarly, for  $ 3j\equiv2i+j(\mod m) $, we get same condition for $ F(x) $ to be triple cycle permutation polynomial over $ F_{2^{m}} $.
\end{proof}
\begin{corollary}
	Let $ m=2k $ and $ F(x)=ax^{2^{k}}+bx $. Then, $ F(x) $ is a triple cycle permutation polynomial over $ F_{2^{m}} $ for all non zero $ a$ and $b $ such that, 
	\begin{equation}
		a,b\in F_{2^{k}},\hspace{0.1cm} and \hspace{0.1cm} b^{2}=ab+1.	 
	\end{equation}
\end{corollary}
\begin{proof}
	It reduces to the $3^{rd}$ case of Theorem (\ref{t5}) with $ j=0 $ gives $ i=k $. 
\end{proof}

One can notice that if $y=x^{2^i}$, then $ F(x)=ax^{2^{i}}+bx^{2^{j}} $ is equivalent $ G(y)=ay+by^{2^{j-i}} $. In [\cite{liu2019triple},Theorem 20], authors proved that $ G(y)=ay+by^{2^{j-i}} $ is a permutation polynomial over $ \mathbb{F}_{2^{m}} $.

\section{Conclusion}\label{sec13}
In this paper, we constructed $ n-$cycle permutation polynomial of a linearised polynomial over $ F_{2^{m}} $. Using the Lemmas (\ref{l1}) and (\ref{l2}) found count for number of monomial $ n-$cycle permutation polynomials over $ F_{2^{m}} $. Additionally,  we constructed $ n-$cycle permutation polynomial for the Kasami power function and Gold power function over $ F_{2^{m}} $.  We  then constructed $ n-$cycle permutation polynomial of the forms $ L(x)+\gamma h(Tr_{q^{m}/q}(x))$, and $  G(x)+\gamma f(x) $, where $ f(x) $ is a Boolean function over $ F_{2^{m}} $. Furthermore, we proved that $ x^{d}+\gamma f(x) $ is quadruple and quintuple with some appropriate conditions over $ F_{2^{m}} $. Finally, we proved that  linear binomial $ ax^{2^{i}}+bx^{2^{j}} $ is a triple cycle permutation polynomial over $ F_{2^{m}} $. In the future, it would be very interesting to obtain more $n-$cycle permutation polynomials of various types using the techniques described in this paper.

\end{document}